\documentclass[10pt]{article}
\usepackage{hyperref}
\usepackage[english]{babel}
\usepackage[latin1]{inputenc}
\usepackage{amsmath,amssymb}
\usepackage{latexsym}
\usepackage{latexsym}
\usepackage{amsfonts}
\usepackage{eucal}
\usepackage{amstext}
\usepackage{amsthm}
\usepackage{amssymb}
\usepackage{amscd}
\theoremstyle{definition}

\newtheorem{defn}{Definition}[section]
\theoremstyle{plain}
\newtheorem{thm}{Theorem}[section]
\newtheorem{prop}[thm]{Proposition}
\newtheorem{lemma}[thm]{Lemma}
\newtheorem{clly}[thm]{Corollary}

\def\prx{\partial_x}
\def\pry{\partial_y}

\def\eqalign#1{\begin{aligned}#1\end{aligned}}
\def\suma{\sum}
\newcommand{\supp}{\operatorname{supp}}

\newcommand{\re}{\mathbb{R}}
\newcommand{\co}{\mathbb{C}}

\newcommand{\nat}{\mathbb{N}}

\newcommand{\Ima}{\mathrm{Im}}

\title{On the existence and analycity of solitary waves solutions to a
two-dimesional Benjamin-Ono equation.}
\author{Germán Preciado López and F\'elix H. Soriano M\'endez}
\linethickness{.12mm}
\begin{document}\maketitle
\begin{abstract}
  We show the existence, regularity and analyticity of solitary waves
  associated to the following equation
\begin{eqnarray*}
 (u_t+u^{p}u_x+ \mathcal H\partial_x^2u+ \lambda \mathcal H\partial_y^2u )_x +\mu u_{yy}=0,
\end{eqnarray*}
where $\mathcal H$ is the Hilbert transform with respect to $x$ and
$\lambda$ and $\mu$ are nonnegative real numbers, not simultaneously
zero.
\end{abstract} \vskip12pt 
In this paper we will use the
following notations: 
\begin{itemize} 
\item $L^2=L^2 (\re ^2)$ 
\item $H^s=H^s(\re ^2)$ 
\item $S=S(\re ^2)$ 
\item $ X^{\frac12}=\{ f \in L^2 | D_x^{\frac12} f \in L^2 \hbox{ and }
D_x^{-\frac12} { f_y} \in L^2  \}$ 
\item $\tilde X^\frac12=\{ f \in L^2 | D_x^\frac 12 f \in L^2 \hbox{ and }
\partial_x^{-1} { f_y} \in L^2  \}$ 
\item $Y=\{ f \in H^1 | \partial_x^{-1} { f_y} \in L^2  \}$ 
\end{itemize}
\section{Introduction} 
In this paper we are interested in the proof of the existence and
analyticity of solitary waves associated to the following equation
\begin{eqnarray}
 (u_t+u^{p}u_x+ \mathcal H\partial_x^2u+ \lambda \mathcal H\partial_y^2u )_x +\mu u_{yy}=0 \label{BOeq},
\end{eqnarray}
where $\mathcal H$ is the Hilbert transform with respect to $x$ and
$\lambda$ and $\mu$ are nonnegative real numbers, not simultaneously
zero. We emphasize 2 striking cases commonly appearing in the
mathematical literature, when $\lambda=1$ and $\mu=0$ and when
$\lambda=0$ and $\mu=1$.  Using Kato's theory, for instance, it can be
proved that the equation \eqref{BOeq} is local well-posed in $H^s\cap
X^\frac12$ and $H^s\cap Y$, for $s> 2$.\par
Observe that
\begin{equation}
E_1(u)=\int_{\re^2}\frac{1}{2}\left((D_x^{\frac{1}2}u)^2+\lambda(D_x^{-\frac12}\partial
_y u)^{2}+ \mu(\partial_x^{-1}\partial
_y u)^{2}\right) + \frac{u^{p+1}}{p(p+1)} dx
\end{equation}
 and \begin{equation} Q(u)=\frac{1}{2}\int_{\re^2}u^2
  dx,\end{equation}
are conserved by the flow of \eqref{BOeq}.\par
This is a two-dimensional case of the Benjamin-Ono equation 
\begin{eqnarray}
\partial_tu+\mathcal H\partial_x^2{u}+u\partial_xu=0,\label{BOeq1}
\end{eqnarray}
which describes certain models in physics about wave propagation in a
stratified thin regions (see \cite{Benjamin} and \cite{Ono}). This last
equation shares with the equation KdV 
\begin{eqnarray} 
u_t+u_x+uu_x+u_{xxx}=0 \label{KdVeq}
\end{eqnarray}
many interesting properties. For example, they both have infinite
conservation laws, they have solitary waves as solutions which are
stable and behave like soliton (this last is evidenced by the
existence of multisoliton type solutions) (see \cite{ablow} and
\cite{matsuno}). Also, the local and global well-posedness was proven
in the Sobolev spaces context (in low regularity spaces inclusive,
see, e.g., \cite{iorio2}, \cite{ponce}, \cite{kenko}, \cite{kochtz}
and \cite{tao})\par
We should note that the equation \eqref{BOeq} is the model of
dispersive long wave motion in a weakly nonlinear two-fluid system,
where the interface is subject to capillarity and bottom fluid is
infinitely deep (see \cite{ablow}, \cite{abloseg} and \cite{kim}). For
this equation, with $\alpha=0$, the local well-posedness was proven in
\cite{guoboling} and the existence of solitary wave solution was
claimed in \cite{spahani}, however their proof is not complete, they
do not present a satisfactory proof of Lemma 3.4 there (Lemma
\ref{lemmaeq9} here). We use interpolation spaces techniques for
this.\par
This paper is organized as follows. In Section \ref{prel} we present
theorem about interpolation of the spaces involved. In Section
\ref{existence} we present the proof of the existence of solitary
waves solution to the equation \eqref{BOeq}, to this we use minimax
theory techniques. Finally, in Section \ref{smoothness}, we shall show
that these solitary waves are analytic (in the real sense) using the
Lizorkin theorem (see \cite{liz}) and ideas developed in \cite{kapi}.
\section{Preliminaries}\label{prel} 
In this section we examine some properties that we shall use later. It
is easy to see that $X^\frac12$ is a Hilbert space with the inner
product defined by $$\langle f,g \rangle_{{\frac12}} =\int_{\re^2} fg
+ D_x^\frac12 fD_x^\frac12 g + \lambda D_x^{-\frac12}
f_yD_x^{-\frac12} g_y+\mu\partial_x^{-1} f_y\partial_x^{-1} g_y
dxdy. $$
\begin{prop}\label{lemab}
$\partial_x S\subset X^\frac12 $.
\end{prop}
Let us recall two important results whose proofs can be found in
\cite{linaresponce}
\begin{lemma}\label{lem1PL}
  Let $s\in(0,1)$,then $H^s({\re}^2)$ is continuously embedded in
  $L^p({\re}^2)$ with $p=\frac{2}{1-s}$. Moreover, for
  $f\in{H}^s({\re}^2)$, $s\in(0,1)$,
\[\|f\|_{L^p}\leq{C}_s\|D_x^sf\|_{L^2}\leq{c}\|f\|_s\] where $$D^{l}
=(-\Delta)^{\frac l2} f = ((2\pi|\xi|)^{l} \hat f)^\vee.$$
\end{lemma}

\begin{lemma}\label{lem2PL}
 Suppose that $D_x^{s_1}f\in{L^2}$ and $D_x^{s_2}f\in{L^2}$. Then, for $s\in[s_1,s_2]$, $D_x^sf\in{L^2}$ and
\begin{eqnarray}
\|D_x^sf\|\leq{C}_s\|D_x^{s_1}f\|^\theta\|D_x^{s_2}f\|^{1-\theta}
\end{eqnarray}
where $\theta=\frac{s_2-s}{s_2-s_1}$.
\end{lemma}

As consequence of these two lemmas we have the following useful
embedding lemma.

\begin{lemma}\label{lemma7}
(a) For $\lambda >0$ and $p\le 2$, there exists  $C>0$, such that, for each $f\in{X^\frac12}$,
\begin{eqnarray}\label{ineq2.0}
\|f\|^{p+2}_{L^{p+2}({\re}^2)}\leq{C}\|f\|^{2-p}\|D_x^{-1/2}\partial_yf\|^{p/2}\|D_x^{1/2}f\|^{\frac{3p}{2}}
\end{eqnarray}
(b) For $\lambda =0$, $\mu>0$ and $p\le \frac43$, there exists  $C>0$, such that, for each $f\in{X^\frac12}$,
\begin{eqnarray}\label{ineq2.01}
  \|f\|^{p+2}_{L^{p+2}({\re}^2)}\leq{C}\|f\|^{\frac{4-3p}3} \|D_x^{1/2}f\|^{\frac{9p+4}6}
  \|\partial_x^{-1} \partial_y
  f\|^{\frac{3p}{2}}
\end{eqnarray}
In particular, for each $f\in{X^\frac12}$

\[\|f\|_{L^{p+2}}\leq{C}\|f\|_{X^\frac12},\] 
where $$\begin{cases}
0\le p\le 2 & \text{ if } \lambda>0,\\
0\le p\le \frac43 & \text{ if } \lambda=0 \text{ and } \mu>0.
\end{cases}
$$
\end{lemma}

\begin{proof}
  By Proposition \ref{lemab}, it is enough show \eqref{ineq2.0} for
  $f=\partial_x \phi$, $\phi\in S$. First, let us suposse
  $\lambda>0$. Lemmas \ref{lem1PL}, \ref{lem2PL} and Hölder inequality
  allow us to show that
\begin{align*}
\|f\|_{p+2}^{p+2}=\int_{{\re}^2)}|f(x,y)|^{p+2}dxdy&\leq{C}\int_{{\re}^2}\|D_x^{p/2(p+2)}f(.,y)\|_0^{p+2}dy\\
&\leq C\int_{{\re}}\|D_x^{1/2}f(.,y)\|^{p}\|f(.,y)\|^2\,dy\\
&\leq{C}\|D_x^{1/2}f\|^{p}\left(
\int_{{\re}}\|f(.,y)\|^{4/(2-p)}\, dy\right)^{(2-p)/2}\\
&\leq{C}\|D_x^{1/2}f\|^{p}\|f\|_0^{2-p}\sup_{y\in{\re}}\|f(.,y)\|^{p}
\end{align*}
On the other hand, for each $y\in{\re}$,
\begin{align*}
\|f(.,y)\|^2&=\int_{\re}f^2(x,y)dx\\
&=2\int_{{\re}}\int_{-\infty}^yf(x,y_1)\partial_yf(.,y_1)dy_1dx\\
&=2\int_{-\infty}^y\int_{{\re}}D_x^{1/2}f(x,y_1)D_x^{-1/2}\partial_{y_1}f(x,y_1)dxdy_1\\
&\leq2\|D_x^{1/2}f\|\|D_x^{-1/2}\partial_{y}f\|
\end{align*}
Hence, we obtain the inequality \eqref{ineq2.0}.\par Now, let us
suppose $\lambda=0$. Proceeding as above, we get 

\begin{equation}
\begin{split}\label{2.12}
  \int_{\re^2} |f|^{p+2} \, d\, x d\, y &\leq c \int_{\re}
  ||D_x^{p/[2(p+2)]}f(\cdot,
  y)||^{p+2}\,d\,y\\
  & \leq c\int_{\re} ||D^{1/2}_xf(\cdot,
  y)||^{(3p+2)/3}||D^{-1/4}_xf(\cdot,y)||^{4/3}
  \, d\, y\\
  & \leq c ||D^{1/2}_xf||_0^{(3p+2)/3} \left( \int_{\re}
    ||D^{-1/4}_xf(\cdot,y)||^{8/(4-3p)}
    \, d\, y\right)^{(4-3p)/6}\\
  & \leq c||D^{1/2}_xf||_0^{(3p+2)/3}||f||_0^{(4-3p)/3}\left( \sup _{y
      \in \re}||D^{-1/4}_xf(\cdot,
    y)||_{L^2(\re)}^{p}\right)
\end{split}
\end{equation}
In this case, for each $y\in{\re}$,
\begin{equation}
\begin{split}
  ||D^{-1/4}_xf(\cdot,y)||_{L^2}^2 &= \int_{\re} \left( D^{-1/4}_xf \right)^2(x,y)\, d \,x\\
  & = 2\int_{\re}\int_{-\infty}^{y}D^{-1/4}_x f(x,y_1)D^{-1/4}_x f_y(x,y_1)\,d\,y_1\,d\,x\\
  & = -2\int_{-\infty}^{y}\int_{\re} D^{1/2}_xf_x(x,y_1)D_x^{-1}f_y(x,y_1)\,d\,x\,d\,y_1\\
  & \leq 2\int_{\re}||D^{-1/2}_xf_x(\cdot, y)||_{L^2(\re)}||\partial_x^{-1}f_y(\cdot, y)||_{L^2(\re)}\,d\,y\\
  & \leq 2||D^{1/2}_xf||_0||\partial_x^{-1}f_y||_0 .
  \end{split}
\end{equation}
Hence it follows \eqref{ineq2.01}. 
\end{proof}

Let $X_0=\{ f\in L^2 |\, \partial_x f,\ \partial_x^{-1} f_{yy} \text{
  and } \partial_x^{-2} f_{yy} \in L^2\}$. $X_0$ is a Hilbert space
with the inner product $$ \langle f,g \rangle_0 =\int_{\re^2} fg
+ \partial_x f \partial_x g +\lambda\partial_x^{-1}
f_{yy}\partial_x^{-1} g_{yy} +\mu\partial_x^{-2} f_{yy}\partial_x^{-2}
g_{yy} dxdy. $$ It is obvious $(X_0, L_2)$ is a compatible couple in
the interpolation theory sense (see \cite{berghlofs}).
\begin{thm}  $X^\frac12 =(X_0, L_2)_{[\frac12]}$.
\end{thm}
\begin{proof}
Let $\phi \in X^\frac 12 $ and $$f(z) =e^{-\delta(z-\frac12)^2}\left(
\left[1+ |\xi| \left( 1+\lambda \left|\frac
  {\eta^2}{\xi^2}\right|+ \mu \left|\frac
  {\eta^2}{\xi^3}\right|\right)\right]^{z-\frac12}\, \hat \phi
\right)^{\vee}.$$ It is easy to see that 
\begin{equation}\label{condf}
\begin{cases} f(z)\in
  L^2,&\text{ for all }0 \le \Ima(z) \le1\\ f\text{ is analytic on }
  0<\Ima(z)<1\\ f(it)\in X_0,&\text{ for all } t\in\re \\ f(1+it)\in
  L^2,&\text{ for all } t\in\re\\ f(z) \to 0,&\text{ as }|\Ima(z)|\to
  \infty,\text{ for } \mathrm{Re} (z)=0\\ f(\frac12)=\phi.
\end{cases}
\end{equation}
Then $\phi\in (X_0, L_2)_{[\frac12]}$ and $\|\phi \|_{[\frac12]}\le
c\|\phi\|_{X^\frac12}$. Now let $\phi\in (X_0, L_2)_{[\frac12]}$,
$ \hat \phi_n=\chi_{|(\xi,\eta)|\le n}\hat \phi $, where
$\chi_{|(\xi,\eta)|\le n}$ is the characteristic function of $|(\xi,\eta)|\le n$ , $\Phi_n (z)=\left((1+|\xi|+\lambda
|\xi|^{-1}|\eta|^2+ \mu |\xi|^{-2}|\eta|^2)^{\frac32-z}\hat \phi_n\right)^\vee$ and
$f$ a function on $0 \le \Ima(z) \le1$ into $L_2$ that satisfies \eqref{condf}.
It is clear that $\Phi_n$ is analytic on $\co$ with values in
$L^2$. Therefore $(f(z), \Phi_n(z))_{L^2}$ is a continuous function on
$0\le \Ima(z)\le 1$  and analytic on $0 < \Ima(z) <1$. Furthermore,
$|(f(it), \Phi_n(it))_{L^2}|\le \|f(it)\|_{X_0}\|
\phi_n\|_{X^\frac12}$ and $|(f(1+it), \Phi_n(1+it))_{L^2}|\le
\|f(1+it)\|_{L^2}\| \phi_n\|_{X^\frac12}$. By three lines lemma, we
have that $$|(f(z), \Phi_n(z))_{L^2}|\le \max(\sup
\|f(it)\|_{X_0},\sup \|f(1+it)\|_{L^2})\| \phi_n\|_{X^\frac12}.$$
Taking $z=\frac12$, we have that  $\| \phi_n\|_{X^\frac12}\le \max(\sup
\|f(it)\|_{X_0},\sup \|f(1+it)\|_{L^2})$, for all $n$. So, by Lebesgue
monotone convergence theorem we have that $\phi\in X_0$.
\end{proof}
\begin{defn}
Let $\Omega$ an open connected set in $\re^2$ and $X$ any $X_0$ or
$X^\frac12$. We denote by
$X(\Omega)$ the set $\big\{ f\in L^2(\Omega)|\, f=g \text{
  for some } g\in X\big\}$. With the norm $$ \|
f\|_{X(\Omega)}= \mathop{\inf_{g|_\Omega =f}}_{g \in
  X^\frac12} \|g \|_{X^\frac12} $$ $X(\Omega)$ is a Banach
space. 
\end{defn}
\begin{lemma}\label{lemma2.5}
Suposse that $\Omega=(a,b)\times (c,d)$ and $\phi$ is a non negative
function $C^\infty$ on $\re$ such that $\supp \phi\subseteq [a,b]$ and
$\int \phi=1$. Then, there exists a constant $C$, depending only on
$\Omega$ and $\phi$, such that for all $f\in L^2_{loc}$ with $(\partial_x^2 f,
\partial_y^2 f) \in
L^2_{loc}$, \begin{equation} \label{ineq2.5}\left\|
  f-\frac1{b-a}\int_a^b f\, dx -\left(x-\frac {a+b}2 \right)\int_a^b f_x\phi \,dx \right\|_{L^2(\Omega)}\le C\|\partial_x^2
  f\|_{L^2(\Omega)},
\end{equation} 
\begin{equation} \label{ineq2.5-6}\left\|
  f_x -\int_a^b f_x\phi \,dx \right\|_{L^2(\Omega)}\le C\|\partial_x^2
  f\|_{L^2(\Omega)}
\end{equation}
and \begin{equation} \label{ineq2.6}\left\|
f_{yy}-\frac1{b-a}\int_a^b f_{yy}\, dx -\left(x-\frac {a+b}2\right)\int_a^b f_{yyx}\phi \,dx \right\|_{L^2(\Omega)}\le C\|\partial_y^2
f\|_{L^2(\Omega)} 
\end{equation}
\end{lemma}
\begin{proof}

First, let us see the following obvious Poincaré inequality generalization.
\begin{lemma}\label{lemma2.7}
Let $a<b$ and $\phi$ be a non negative continuous function on $[a,b]$ such that $\int
\phi=1$. Then, for all $f\in L^p[a,b]$ with $f'\in
L^p[a,b]$, $$\left\|f-\int_a^b f\phi\, dx \right\|_{L^p[a,b]}\le
C\|f'\|_{L^p[a,b]},$$ where $C$ depends only on $[a,b]$ and $p$.
\end{lemma}
\begin{proof}
For $x\in [a,b]$, $$ \left|f(x)-\int_a^b f(\xi)\phi(\xi)\, d\xi \right|= \left|\int_a^b
\int_\xi^x f'(s)\, ds \phi(\xi)\, d\xi\right|\le \|f'\|_{L^1[a,b]}. $$
The lemma follows immediately from this inequality.
\end{proof}
By Poincare's inequality and lemma above,
\begin{multline*}
\int_a^b \left|f(x,y)-\frac1{b-a}\int_a^b f\  -\left(x-\frac {a+b}2
  \right)\int_a^b f_x\phi \,dx , dx \right|^2\, dx\le \\
\le C^2 \int_a^b \left|f_x(x,y)-\int_a^b f_x\phi \, dx \right|^2\,
dx\le \\\le C^2\int_a^b |\partial_x^2 f(x,y)|^2\,
dx.
\end{multline*}
This inequality shows \eqref{ineq2.5}.\par \eqref{ineq2.5-6} is an
immediate consequenece of Lemma \ref{lemma2.7}. Now we shall prove
\eqref{ineq2.6}. By the Cauchy-Schwarz inequality we have that 
\begin{equation}\label{ineq2.10-1}
\left|\frac1{b-a}\int_a^b f_{yy}\, dx\right|^2\le \frac1{b-a}\int_a^b
|f_{yy}|^2\, dx. $$ Therefore
$$ \left\|\frac1{b-a}\int_a^b f_{yy}\, dx\right \|_{L^2(\Omega)}\le
\|\partial^2_y f\|_{L^2(\Omega)}.
\end{equation}
Additionally, observe that
\begin{equation}\label{ineq2.11}
 \left|\int_a^b f_{xyy}(x,y) \phi(x)\, dx \right|
=\left|\int_a^b f_{yy}(x,y) \phi_x(x)\, dx \right|\le
\|f_{yy}(\cdot,y)\|_{L^2[a,b]}\|\phi_x\|_{L^2[a,b]}.
\end{equation}
\eqref{ineq2.10-1}, \eqref{ineq2.11} and the triangle inequality imply
\eqref{ineq2.6}.
\end{proof}
 
\begin{lemma}\label{lemma2.6}
 Let $\Omega=(a,b)\times(c,d) $. There exists an extension operator
 $E: X^0(\Omega) \to  X^0$,
  i.e., there exists a bounded linear operator $E$ from $ X^0(\Omega)$
  to $ X^0$ such that, for any $u\in  X^0(\Omega)$, $Eu=u$ in $\Omega$, $\|Eu\|_{L^2}\le C\|u\|_{L^2(\Omega)}$ and $\|Eu\|_{ X^0}\le
  C\|u\|_{ X^0(\Omega)}$, where $C$ depends only on $\Omega
$.
\end{lemma}
\begin{proof}
Let $u\in  X_0(\Omega)$. Without loss generality, we can suposse
that $u=\partial_x^2 f$ in $\Omega$, for some $f\in S(\re^2)$ with
$\|\partial_x^2 f\|_{ X^0}\le 2\|u\|_{ X^0(\Omega)}$. Let us take $f_0=
f-\frac1{b-a} \int_a^b f\, dx -\left(x-\frac {a+b}2 \right)\int_a^b
f_x\phi \, dx$. It is obviuos that $u=\partial_x^2 f_0$ in $\Omega$.
Now consider $f_1$ defined on $[2a-b, 2b-a]\times [c,d]$ by $$
f_1(x,y)= \begin{cases} f_0(x,y) & \text{ if } x\in [a,b]\\
\sum_{i=1}^4a_i f_0(\frac {i+1}ib -\frac 1i x,y)  & \text{ if } x\in [b,2b-a]\\
\sum_{i=1}^4a_i f_0(\frac {i+1}ia -\frac 1i x,y)   &\text{ if } x\in [2a-b,a],
\end{cases} 
$$  
where 
$$ 
\begin{aligned}
a_1 +a_2 +a_3+a_4&=1\\ 
a_1 +\frac {a_2}2 +\frac {a_3}3+\frac {a_4} 4&=-1\\
a_1 +\frac {a_2}4 +\frac {a_3} 9+\frac {a_4} {16}&=1\\
a_1 +\frac {a_2}8 +\frac {a_3} {27}+\frac {a_4} {64}&=-1
\end{aligned}
$$ Clearly $ f_1$ is a $C^3 $ function on $[2a-b, 2b-a]\times [c,d]$
and satisfy
\begin{equation}\label{ineq2.8}
\| \partial^\alpha f_1\|_{L^2([2a-b,
    2b-a]\times [c,d])}\le C\| \partial^\alpha f_0\|_{L^2(\Omega)},
\end{equation} 
for all $\alpha \in \nat^2$ with $|\alpha |\le 3$.
In the same way, from $f_1$, we can define a $C^3$ function $f_2$ on $\tilde \Omega
=[2a-b, 2b-a]\times [2c-d,2d-c]$ such that 
\begin{equation}\label{ineq2.9}
\| \partial^\alpha f_2\|_{L^2(\tilde \Omega)}\le 9\| \partial^\alpha f_0\|_{L^2(\Omega)},
\end{equation} 
for all $\alpha \in \nat^2$ with $|\alpha |\le 3$.
Now, let $\eta$ a $C^\infty$ function in $\re^2$ such that $\eta\equiv 1$ in $
\Omega$ and $0$ out of $\tilde \Omega$, and let $Eu= \partial_x^2(\eta f_2)$
in $\tilde \Omega$ and $0$ in $\re^2-\tilde \Omega$. From
\eqref{ineq2.9} and Lemma \ref{lemma2.5} follows that 
$Eu=u$ in $\Omega$, $\|Eu\|_{L^2}\le C\|u\|_{L^2(\Omega)}$ and $\|Eu\|_{X^0}\le
  C\|u\|_{X^0(\Omega)}$, where $C$ depends only on $\Omega
$ and $\phi$.
\end{proof}
\begin{clly}
If $\Omega =(a,b)\times (c,d)$ then $X^\frac 12(\Omega)=[L^2(\Omega),
  X^0(\Omega)]_{[\frac 12]}$.
 \end{clly}
\begin{proof} 
  It is enough to observe that $E$ defined in Lemma \ref{lemma2.6}
  can be see as a coretract of the restriction operator from $(X^0,L^2)$
  to $(X^0(\Omega),L^2(\Omega))$. Then, the corollary follows from
  Theorem 1.2.4 in \cite{triebel}
\end{proof} 
\begin{thm}\label{thm2.12}
Suposse that $\{\Omega_i\}_{i\in\nat}$ is a cover of $\re^2$, where
each $\Omega_i$ is an open cube with edges parallel to the coordinate
axis and side-length $R$, and such that each point in $\re^2$ is
contained in at most 3 $\Omega_i$'s. Then
\begin{equation}\label{ineq2.10}
\sum_{i=0}^\infty \|u\|_{X(\Omega_i)}^2 \le C\|u\|_{X}^2,
\end{equation}
for all $u\in X^\frac 12$.
\end{thm}
\begin{proof}
  Proceeding as in the proof of Lemma \ref{lemma2.6} we can show
  that $$\|E_i u\|_{X^0}^2 \le C\int_{\Omega_i} u^2 + \partial_x u^2
  +\lambda \partial_x^{-1}\partial_y^2u^2 +
  \mu \partial_x^{-2}\partial_y^2u^2\,dx ,$$ where $E_i$ is the extension
  operator from $X^0(\Omega_i)$ to $X^0$. It is easy to check that $C$
  depends only on length of $x$-side of $\Omega_i$. Then $C$ is
  independent of $i$. Since $$ \|u\|_{X^0(\Omega_i)}\le \|E_i
  u\|_{X^0},$$ for all $i$, we get $$ \sum_{i=0}^\infty
  \|u\|_{X^0(\Omega_i)}^2 \le C\sum_{i=0}^\infty\int_{\Omega_i} u^2
  + \partial_x u^2 +\lambda \partial_x^{-1}\partial_y^2u^2 +
  \mu \partial_x^{-2}\partial_y^2u^2\, dxdy\le 3C
  \|u\|_{X^0}^2.$$ Also, it is obvious that $$ \sum_{i=0}^\infty
  \|u\|_{L^2(\Omega_i)}^2 \le 3 \|u\|_{L^2}^2.$$ Then the operator $u
  \mapsto (u_{\Omega_i})_{i\in \nat}$ ($u_{\Omega_i}$ is the
  restriction of $u$ to ${\Omega_i}$) is continuous from $L^2$ to $
  \ell^2 (L^2(\Omega_i))$ and from $X^0$ to $ \ell^2 (X^0(\Omega_i))$. By
  Theorem 1.18.1 in \cite{triebel}, we have that the operator $u \mapsto
  (u_{\Omega_i})_{i\in \nat}$ is continuous from $X^\frac12$ to $ \ell^2
  (X^\frac12(\Omega_i))$. Thence we obtain \eqref{ineq2.10} for
  $X=X^\frac12$. The proof of \eqref{ineq2.10} with $X=\tilde
  X^\frac12$ is completely analogous.
\end{proof}
\begin{lemma}\label{lemmaEC}
The embedding $X^\frac12\hookrightarrow L_{loc}^p({\re}^2)$ is
compact, if $$\begin{cases} 0\le p<4 &\text{ if } \lambda >0\\  0\le
  p<\frac43  &\text{ if } \lambda= 0.
\end{cases}
$$
In other words, if $(u_n)$ is a bounded sequence in $X^\frac12$ and
$R>0$, there exists a subsequence $(u_{n_k})$ of $(u_n)$ which
converges strongly to $u$ in $L^p (B_R)$.
 \end{lemma}

\begin{proof}
  We prove the lemma when $\lambda >0$, the proof when $\lambda=0$
  is just to make some obvious modifications. Suppose that
  $(u_n)_{n=1}^{\infty}$ is a bounded sequence in $X^\frac 12$.  Let
  $\Omega_R$ be the cube with center at the origen and edges parallel
  to the coordinate axis and length $R$, and let $E_R$ the extension
  operator from $L^2( \Omega_R)$ to $L^2$ as in proof Lemma
  \ref{lemma2.6}. By interpolation, $E_R$ is a continuous operator from
  $X^\frac12(\Omega)$ to $X^\frac12$. Also, it is easy to observe that
  $E_R(u)$ is $0$ out of $\Omega_{3R}$, for all $u\in X^\frac12$,
  where $\Omega_{3R}$ is the cube with center at the origen and edges
  parallel to the coordinate axis and length $3R$. Because $u=E_R(u)$
  in $\Omega$, withuot loss of generality, we can assume that
  $u_n=E_R(u_n)$, for all $n$. Now, since $u_n$ is bounded in
  $X^\frac12$, we can also suposse that $u_n\rightharpoonup u$ in
  $X^\frac12$, and replacing, if necessary, $u_n$ by $u_n-u$, we can
  assume that $u=0$ too.\par Let
\begin{align*}
Q_1&=\{(\xi,\eta)\in{\re}^2/|\xi|\leq\rho,|\eta|\leq\rho\}\\
Q_2&=\{(\xi,\eta)\in{\re}^2/|\xi|>\rho\}\\
Q_3&=\{(\xi,\eta)\in{\re}^2/|\xi|<\rho,|\eta|>\rho\}
\end{align*}
Then  ${\re}^2=\bigcup_{i=1}^3Q_i$  and $Q_i\bigcap{Q}_j=\emptyset$,  $i\neq{j}$.
For $\rho>0$, there holds
\[\int_{\Omega_{3R}}|u_n(x,y)|^2dxdy=\int_{{\re}^2}|\widehat{u}_n(\xi,\eta)|^2d\xi{d}\eta=\sum_{i=1}^3\int_{Q_i}|\widehat{u}_n(\xi,\eta)|^2d\xi{d}\eta\]
It is clear that
\[\int_{Q_2}|\widehat{u}_n(\xi,\eta)|^2d\xi{d}\eta=\int_{Q_2}\frac{1}{|\xi|}|\widehat{D^{1/2}_x{u}_n}(\xi,\eta)|^2d\xi{d}\eta
\leq\frac{C}{\rho}\|D_x^{1/2}u_n\|_0^2,\]
and
\[\int_{Q_3}|\widehat{u}_n(\xi,\eta)|^2d\xi{d}\eta=\int_{Q_3}\frac{|\xi|}{|\eta|^2}|\widehat{D^\frac12_x\partial_yu}_n(\xi,\eta)|^2d\xi{d}\eta.
\]
Therefore, for any $\epsilon$, there exists $\rho>0$ large enough such that
\[\int_{Q_2}|\widehat{u}_n(\xi,\eta)|^2d\xi{d}\eta+\int_{Q_3}|\widehat{u}_n(\xi,\eta)|^2d\xi{d}\eta\leq\epsilon/2.\]
Since, by the fact that $u_n\rightharpoonup0$ in $L^2({\re}^2)$,
\[\lim_{n\rightarrow0}\widehat{u}_n(\xi,\eta)=\lim_{n\rightarrow0}\int_{\Omega_{3R}}u_n(x,y)e^{-i(x\xi+y\eta)}dxdy=0,\]
and $|\widehat{u}(\xi,\eta)|\leq \|u_n\|_{1}$, the Lebesgue dominated
convergence theorem guarantees that $$\int_{Q_1}
|\widehat{u}_n(\xi,\eta)|^2d\xi{d}\eta=0$$ as $n\rightarrow{\infty}$.
Hence $u_n\rightarrow{0}$ in $L_{loc}^2({\re}^2)$. By Lemma
\ref{lemma7}, $u_n\rightarrow{0}$ in $L_{loc}^p({\re}^2)$ if
$2\leq{p}< 4$.

\end{proof}

\begin{lemma}\label{lemmaeq9}
 If $(u_n)$ is bounded in $X^\frac12$ and
\begin{equation}\label{eq9}
\lim_{n\to\infty}\sup_{(x,y)\in{\re}^2}\int_{B(x,y;R)}|u_n|^2dxdy=0,
\end{equation}
as $n\to\infty$, then $u_n \to{0}$ in $L^p({\re}^2)$ for 
$$\begin{cases} 2< p<4 & \text{ if } \lambda >0 \\ 2< p<4/3 & \text{ if } \lambda =0
\end{cases}
. $$
\end{lemma}
\begin{proof}
  Suposse $\lambda>0$ ($\lambda=0$ follows in the same way). Let $2< s<
  4$ and let $\Omega_R$ be the cube with center at the origen, edges
  parallel to the coordinate axis and side-length $R$. Then, by Hölder
  inequality and Lemma \ref{lemma7}, we have that

\begin{align*}
\|u_n\|_{L^s{((x,y)+\Omega_R)}}&\leq\|u_n\|^{1-\vartheta}_{L^2_{((x,y)+\Omega_R)}}\|u\|^{\vartheta}_{L^{4}_{((x,y)+\Omega_R)}}\\
&\leq\|u_n\|^{1-\vartheta}_{L^2_{((x,y)+\Omega_R)}}\|u_n\|^{\vartheta}_{X^\frac12_{((x,y)+\Omega_R)}},
\end{align*}
 where $\vartheta =\frac{2(s-2)}{s}$. 
Choosing $s$ such that $\frac{\vartheta{s}}{2}=1$, i.e.,
$s=3$, there holds

\[\int_{(x,y)+\Omega_R)}|u_n|^{3}dxdy\leq{C}\|u_n\|_{L^2((x,y)+\Omega_R)}\|u_n\|^{2}_{X^\frac12_{((x,y)+\Omega_R)}},\]
Now, covering ${\re}^2$ by cubes with edges parallel to the coordinate
axis and side-length $R$ in such a way that each point of ${\re}^2$ is
contained in at most 3 of these cubes, by Theorem \ref{thm2.12}, we get 
$$
\int_{{\re}^2}|u_n|^{3}dxdy\leq C\sup_{(x,y)\in{\re}^2}\|u_n\|_{L^2((x,y)+\Omega_R)}\|u_n\|^{2}_{X^\frac12}
$$
Since $u_n$ is bounded in $X^\frac12$ and satisfies \eqref{eq9},
$u_n\to{0}$ in $L^{3}({\re}^2)$. Because $2<3<4$, the Hölder
inequality implies that $u_n\to{0}$ in $L^p({\re}^2)$, for all $2<p<4$.
\end{proof}
The following lemma gives us a minimax principle and is an immediate
consequence of Theorem 2.8 in \cite[pg. 41]{willem}
\begin{lemma}\label{existPSseq} 
Suppose $X$ is a Banach space and $\Phi\in{C}^1(X,{\re})$ satisfies the
following properties:
 \begin{enumerate}
\item $\Phi(0)=0$, and there exists $\rho>0$, such that $\Phi|_{\partial{B}_\rho(0)}\geq\alpha>0$.
\item There exists $\beta\in{X}$ $\backslash$ $\overline{B}_{\rho}(0)$ such that $\Phi(\beta)\leq0.$
 \end{enumerate}
Let $\Gamma$ be the set of all paths which connects $0$ and $\beta$, i.e.,
\[\Gamma=\{g\in{C}([0,1],X)\,|\,g(0)=0, g(1)=\beta\},\]
and
\begin{equation}
c=\inf_{g\in\Gamma}\max_{t\in[0,1]}\Phi(g(t)).
\end{equation}
Then $c\geq\alpha$ and $\Phi$ possesses a Palais-Smale sequence at
level $c$, i.e., there exists a sequence $(u_n)$ such
that $\Phi(u_n)\to{c}$ and $\Phi'(u_n)\to{0}$ as $n\to\infty$.

\end{lemma}
\section{Existence of Solitary Waves}\label{existence} 

 If $\phi(x-ct, y)$ is a solitary wave solution solution to
 \eqref{BOeq}, then \begin{equation} (-c\partial
   _x{\phi}+\phi^{p}\partial_x{\phi} +\mathcal H(\partial_x^2{\phi}+ \lambda \partial
   _{y}^2{\phi}))_x + \mu \pry^2 \phi
   =0. \label{(3.1)} \end{equation} 
If $\phi\in X^\frac12$, we can write~(\ref{(3.1)}) as 
\begin{equation}
  -c\phi+\mathcal H\partial_x\phi+\lambda \mathcal H\partial _x^{-1}\partial^2
  _y{\phi}+ \mu \partial _x^{-2}\partial^2
  _y{\phi}+\frac 1{p+1}\phi^{p+1}=0.
\label{(3.2)}
\end{equation}
where the term on the right hand is in $\left( X^\frac12 \right)^*$, the
topological dual of $X^\frac12$. 
 Then $\phi$ is a critical point of the functional $\Phi$ on
 $X^\frac12$ defined as
$$
\Phi(\phi)=\int_{\re^2}\frac12\left(c\phi^2+(D_x^{\frac12}\phi)^2+\lambda
  (D_x^{-\frac12}\partial _y \phi)^{2}+ \mu (\partial_x^{-1}\partial _y \phi)^{2}  \right)-\frac{\phi^{p+2}}{(p+1)(p+2)}\, dxdy.
$$ 
Let us see that $\Phi$ satisfies the conditions of the Lemma
\ref{existPSseq}. It is obvious that $\Phi$ is a $C^1$ functional for
$0<p\le 2$. $\Phi(0)=0$ and, since $$ \Phi(\phi)\ge
\frac{\min\{c,1\}}2\|\phi\|_{X^\frac12} -\frac
     {|\phi|^{p+2}}{(p+1)(p+2)}, $$ by Lemma \ref{lemma7}, there exist
     a $\rho$ such that $$\inf_{\partial B_\rho(0)}\Phi=\alpha>0,$$
     which shows $1)$. Now, for $\vartheta \in \re$ and $u\in
     X^\frac12$, $$\Phi(\vartheta u)= \vartheta^2
     \left(\Phi(u)+\int_{\re^2}\frac{u^{p+2}}{(p+1)(p+2)}\, dxdy
     \right) -\vartheta^{p+2}\int_{\re^2} \frac{u^{p+2}}{(p+1)(p+2)}\, dxdy.$$
Then, taking $u$ fixed and $\vartheta$ large enough, we have $2)$ with
$\beta =\vartheta u$. So, we have shown the following lemma.
\begin{lemma}\label{existPSseq1}
Let $\Phi$, $\alpha$ and $\beta$ be defined as above and let $ \Gamma$ and $c$
be defined as Lemma \ref{existPSseq}. Then, there exists a sequence
$(\phi_n)$ such that $\Phi(\phi_n)\to c$ and $\Phi'(\phi_n)\to 0$.
\end{lemma}
 Now, we can prove the following theorem.
\begin{thm}
 \eqref{(3.1)} has nontrivial solutions in $X^\frac12$. 
\end{thm}
\begin{proof}
It is enough to show that $\Phi$ have non-zero critical points in
$X^\frac12$. By Lemma \ref{existPSseq1}, there exists a Palais-Smale
sequence $(\phi_n)$ at level $c$ of $\Phi$. Therefore, $$c+1\ge
\Phi(\phi_n)-\frac {\langle
  \Phi'(\phi_n),\phi_n\rangle_{X^\frac12}}{p+2} \ge (\frac12
-\frac1{p+2})\min\{1,c\} \|\phi_n\|^2_{X^\frac12}, $$ for $n$ big
enough. Hence $(\phi_n)$ is bounded in $X^\frac12$. Considering
that $$ 0<c =\lim_{n\to \infty} \Phi(\phi_n) -\frac12 \langle
\Phi'(\phi_n),\phi_n\rangle_{X^\frac12} =\lim_{n\to \infty} \frac
p{2(p+2)(p+1)} \int_{\re^2}\phi_n^{p+2}\, dx dy,$$ the Lemma
\ref{lemmaeq9} implies that
$$ \delta =\limsup_{n\to\infty} \sup_{(x,y)\in \re^2}
\int_{(x,y)+\Omega_1} \phi_n^2 \, dxdy>0 .$$ Then, passing to a
subsequence if necessary, we can assume that there exists a sequence
$(x_n,y_n)$ in $\re$ such
that \begin{equation}\label{eq3.1}\int_{(x_n,y_n)+\Omega_1} \phi_n^2
  \, dxdy>\delta/2,\end{equation} for $n$ big enough. Let $\tilde
\phi_n=\phi_n(\cdot+(x_n,y_n))$. Then, again passing to a subsequence
if necessary, we can assume that, for some $\phi\in X^\frac12$,
$\tilde\phi_n\rightharpoonup \phi$ in $X^\frac12$. In view of
\eqref{eq3.1}, for $n$ large enough, and Lemma \ref{lemmaEC}, $\phi\ne
0$. The Lemma \ref{lemmaEC} and the continuity of the function $u\to
u^{p+1}$ from $L^{p+2}$ to $L^{\frac{p+2}{p+1}}$, in any measure
space, imply that $$ \langle \Phi'(\phi),w\rangle_{X^\frac12}
=\lim_{n\to\infty} \langle
\Phi'(\tilde\phi_n),w\rangle_{X^\frac12}=0.$$ This shows this theorem.
\end{proof}
\section{Smoothness of solitary wave}\label{smoothness}
In this section we shall proof that the solitary wave solution of
\eqref{BOeq} is $C^\infty$. 
\begin{thm}
Let $p=1$. If $\phi \in X^\frac12$ is solution to \eqref{(3.1)}, $\phi\in
H^\infty=\bigcap_0^\infty H^n$. Moreover, $\phi$ is analytic.
\end{thm}
\begin{proof}
  Suposse first that $\mu=0$. In this case, without loss of
  generality, we can suposse that $\alpha=1$. By Lemma \ref{lemma7},
  $\phi\in L^4$. In particular, $\frac12\phi^2+c\phi\in L^2$. Now,
  from \eqref{(3.1)}, we have
\begin{equation}\label{(4.1)}
\Delta \phi =\mathcal H\partial_x (\frac{ \phi^2}2 -c\phi).
\end{equation}
Then, the Plancherel theorem implies that $\phi\in H^1$. So, by
Sobolev embedding theorem, $\frac12\phi^2+c\phi\in L^p$, $2\le
p<\infty$. Since the Hilbert transform is bounded from $L^p\to L^p$
and, by Lizorkin theorem (see \cite{liz}), $\frac {\xi^2}{\xi^2+\eta^2}$ and $\frac
{\xi\eta}{\xi^2+\eta^2}$ are $L^p$ multipliers, from \eqref{(3.1)}, we
have that $\phi_x $ and $\phi_y\in L^p$. Whence, again \eqref{(3.1)}
implies that $\phi\in H^2$. The theorem follows once we have observed
that if $\phi\in H^n$ then $\phi\in H^{n+1}$, for $n\ge 2$. This last
affirmation follows from \eqref{(3.1)}, the fact that $H^n$ is a
Banach algebra, for $n\ge 2$, and Plancherel theorem.\par 
Suppose now $\gamma\mu\ne 0$. Without loss of generality, 
we can suppose also that $\gamma=\mu=1$. So, \eqref{(3.1)}
becomes in 
\begin{equation}\label{(4.2)}
  \mathcal H\prx^3 \phi+ \mathcal H\prx\pry^2 \phi -\pry^2 \phi
  =-\partial_x^2 (\frac{ \phi^2}2 - c\phi).
\end{equation}
From here, thanks to Lizorkin theorem, we have that $\frac
{\xi^3}{|\xi|^3+|\xi|\eta^2+\eta^2}$ and $\frac {\xi^2\eta} {|\xi|^3+
  |\xi|\eta^2+\eta^2}$ are multipliers in $L^p$, $1<p<\infty$. From
now on it is just follow the steps in the previous case. The case
$\gamma=0$ is was done in \cite{spahani}.\par
To see the analyticity of $\phi$ it is enough to prove that
\begin{equation} \|\partial^\alpha \phi\|_{H^2}\le 
C|\alpha|!\Big(\frac{R}{2}\Big)^{|\alpha|},\label{cond-fi}\end{equation} 
for some $R>0$ and for all $\alpha\in \nat^2$. We shall show that
there exists $R>0$ such that for all $\alpha\in \nat^2$
\begin{equation}\label{ine.ayuanal}\|\partial^\alpha \phi\|_{H^2}\le
C\frac{(|\alpha|-1)!}{(|\alpha|+1)^s}\Big(\frac{R}{2}\Big)^{|\alpha|-1},
\end{equation} where $s>1$. We see this by induction. For $|\alpha|=1$
the inequality~(\ref{ine.ayuanal}) is  obvious; it is sufficient to choose $C$ large
enough. Suppose now that~(\ref{ine.ayuanal})  is valid for $|\alpha|=1,
\cdots ,n$ and $R$ (that we shall conveniently choose later). From
equation \eqref{(3.1)} we have that 
\begin{equation}\label{(4.3)}
  \prx^2 \phi+\alpha \pry^2 \phi - \gamma\mathcal H\prx^{-1}\pry^2 \phi
  =\mathcal H\partial_x (\frac{ \phi^2}2 - c\phi).
\end{equation}
Applying $\partial^\alpha$ on both sides of the equation and making the
inner product in $H^2$ with $\partial^\alpha\phi$ in the last equation,
we can show that
\begin{equation}\label{ine.ay1} \|\nabla\partial^\alpha\phi\|_{H^2}\le
C\left\|\partial^ \alpha \Big(\frac
{\phi^2}2-c\phi)\Big) \right\|_{H^2}.\end{equation} 
For finishing the theorem's proof we need the following lemma. 
\begin{lemma}\label{prop.ayu}

(a) If $f$ and $\phi\in C^\infty(\re) $, then
$$\partial^\alpha \Big(f(\phi)\Big) =\suma_{j=1}^{|\alpha|}
\frac{f^{(j)}(\phi)}{j!}\mathop{\suma_{ \alpha_1+\cdots+\alpha_j=\alpha}}
\limits_{|\alpha_i|\ge 1,\ \forall\  1\le i\le j} \frac{\alpha!}{\alpha_1!
\cdots \alpha_j!} \partial^{\alpha_1}\phi \cdots \partial^{\alpha_j}\phi. $$ 

(b) For each $(n_1,\ldots,n_j)\in \nat^j$ we have $$|\alpha|!=\mathop{\suma_{
\alpha_1+\cdots+\alpha_j=\alpha}} \limits_{|\alpha_i|=n_i,\ \forall\  1\le i\le
j}\frac{\alpha!|\alpha_1|! \cdots |\alpha_j|!}{\alpha_1! \cdots \alpha_j!}.$$

(c) For $s>1$ there exists $C_2$ such that for all $j\ e\  k\in\nat$
$$\suma_{k_1+\cdots+k_j=k}\frac{1}{(k_1+1)^s\cdots (k_j+1)^s} \le
\frac{C_2^{j-1}}{(k+1)^s}$$ \end{lemma} Now we return to the proof of
the theorem. Part \emph{(a)} of Lemma \ref{prop.ayu}, inequality
\eqref{ine.ay1} and the fact $H^2$ is a Banach algebra imply that $$
\|\nabla\partial^\alpha\phi\|_{H^2} \le C_1\suma_{j=1}^{2}
\mathop{\suma_{ \alpha_1+\cdots+\alpha_j=\alpha}}
\limits_{|\alpha_i|\ge 1,\ \forall\ 1\le i\le j}
\frac{\alpha!}{\alpha_1! \cdots \alpha_j!}
\|\partial^{\alpha_1}\phi\|_ {H^2} \cdots
\|\partial^{\alpha_j}\phi\|_{H^2}.$$ By the induction hypothesis
and part \emph{(b)} of the same lemma, we have $$
\eqalign{\|\nabla\partial^\alpha\phi\|_{H^2} \le &
  C_1\suma_{j=1}^{2}\hskip -1mm C^jA^{|\alpha|-j}\hskip -8mm
  \mathop{\suma_{ n_1+\cdots+n_j=|\alpha|}} \limits_{n_i\ge 1,\
    \forall\ 1\le i\le j} \mathop{\suma_{
      \alpha_1+\cdots+\alpha_j=\alpha}} \limits_{|\alpha_i|=n_i,\
    \forall\ i} \frac{\alpha!}{\alpha_1! \cdots \alpha_j!} \frac
  {(|\alpha_1|\hskip -1mm-\hskip -1mm1)! \cdots (|\alpha_j|\hskip
    -1mm-\hskip -1mm1)!}{(|\alpha_1|\hskip -1mm+\hskip
    -1mm1)^s\cdots(|\alpha_j|\hskip -1mm+\hskip -1mm1)^s} \cr \le
  &C_1\suma_{j=1}^{2}{\tilde C}^jA^{|\alpha|-j} \mathop{\suma_{
      n_1+\cdots+n_j=|\alpha|}} \limits_{|n_i|\ge 1,\ \forall\ 1\le
    i\le j} \frac{|\alpha|!}{(n_1+1)^{s+1}\cdots(n_j+1)^{s+1}},}$$
where $A=\frac{R}{2}$, and from this inequality and part {\em (c)} of
Lemma \ref{prop.ayu}, we obtain that
$$\|\nabla\partial^\alpha\phi\|_{H^2}\le 
C_1\frac{|\alpha|!}{(|\alpha|+2)^s}A^{|\alpha|}\suma_{j=1}^{2}({\tilde
C}C_2)^jA^{-j}.$$ Now we can choose $R$. We take $A$
large enough such that $C_1\suma_{j=1}^{2}({\tilde
C}C_2)^jA^{-j}\le C$. It is clear that this choice does not depend on
$\alpha$. Therefore, with $R=2A$,
$$\|\nabla\partial^\alpha\phi\|_{H^2}\le 
C\frac{|\alpha|!}{(|\alpha|+2)^s}\Big(\frac{R}{2}\Big)^{|\alpha|},$$
that shows~(\ref{ine.ayuanal}). This completes the proof 

\end{proof}
\bibliographystyle{acm}
\bibliography{mybiblio}
\end{document}